\documentclass[11pt,english,a4paper]{amsart}
\pdfoutput=1

\usepackage[utf8]{inputenc}
\usepackage[T1]{fontenc}
\usepackage{lmodern}
\usepackage{babel}
\usepackage{amsmath,amssymb,amsthm}
\makeatletter
\@ifclassloaded{amsproc}{}{%
  \usepackage[marginratio={1:1,1:1},totalwidth=408pt,totalheight=612pt]{geometry}}
\makeatother
\usepackage{microtype}
\usepackage{hyperref}
\usepackage{url}

\newif\ifanonymous
\anonymousfalse

\ifanonymous
\hypersetup{hidelinks,
  pdftitle={Selflessness for twisted group C*-algebras of amenable groups and their inclusions}}
\else
\hypersetup{
  pdftitle={Selflessness for twisted group C*-algebras of amenable groups and their inclusions},
  pdfauthor={Tron Omland}}
\fi

\theoremstyle{plain}
\newtheorem{theorem}{Theorem}[section]
\newtheorem{lemma}[theorem]{Lemma}
\newtheorem{proposition}[theorem]{Proposition}
\newtheorem{corollary}[theorem]{Corollary}
\newtheorem{mainthm}{Theorem}

\theoremstyle{definition}
\newtheorem{definition}[theorem]{Definition}
\newtheorem{remark}[theorem]{Remark}
\newtheorem{example}[theorem]{Example}

\newcommand{\C}{\mathbb{C}}
\newcommand{\N}{\mathbb{N}}
\newcommand{\Z}{\mathbb{Z}}
\newcommand{\T}{\mathbb{T}}
\newcommand{\U}{\mathcal U}
\newcommand{\Cr}{C^*_r}
\newcommand{\Zst}{\mathcal{Z}}
\DeclareMathOperator{\rk}{rank}
\DeclareMathOperator{\ord}{ord}
\DeclareMathOperator{\supp}{supp}

\title[Selfless twisted group $C^*$-algebras]{Selflessness for twisted group $C^*$-algebras of amenable groups and their inclusions}

\ifanonymous\else
\author[Omland]{Tron Omland}
\address{Norwegian National Security Authority (NSM) \and Department of Mathematics, University of Oslo, Norway}
\email{tron.omland@gmail.com}
\fi


\date{September 29, 2026}

\begin{document}

\ifanonymous\else
\dedicatory{Dedicated to Erik B\'edos on the occasion of his 70th birthday.}
\fi

\begin{abstract}
For a discrete amenable group $G$ with a two-cocycle $\sigma$, we study when the twisted group $C^*$-algebra $\Cr(G,\sigma)$ is selfless, in the sense of Robert. For a countably infinite abelian group $G$, this holds exactly when the commutator bicharacter of $\sigma$ is nondegenerate, that is, under Kleppner's condition. The nuclear dimension is then zero if $G$ is torsion and one otherwise. We show that root-of-unity-valued cocycles can give selfless algebras of nuclear dimension one, using slow dimension growth. The characterization extends to countably infinite virtually abelian or locally finite FC-hypercentral groups and to infinite finitely generated virtually nilpotent groups. Further, if $G$ is countable and $H\trianglelefteq G$ is infinite, the inclusion $\Cr(H,\sigma')\subseteq\Cr(G,\sigma)$ is selfless precisely when $\Cr(H,\sigma')$ is selfless and $(H\leq G,\sigma)$ satisfies the relative Kleppner condition. As applications, we obtain selfless twisted group $C^*$-algebras of the lamplighter group and of the wreath product $\Z\wr\Z$.
\end{abstract}

\maketitle

\section{Introduction}

A tracial $C^*$-probability space $(A,\rho)$ is \emph{selfless} \cite{Robert} if the first-factor embedding $(A,\rho)\hookrightarrow(A,\rho)\star(A,\rho)$ into the reduced free product is existential (see Definition~\ref{d:selfless}). Selflessness is a free-probabilistic regularity property. It forces $A$ to be simple and to have a unique tracial state, strict comparison, and stable rank one. Conversely, every infinite-dimensional simple unital $C^*$-algebra with a unique tracial state and strict comparison is selfless \cite{RobertSC}. The reduced (twisted) group $C^*$-algebras of large classes of non-amenable groups are selfless, and so are many reduced crossed products, see for instance \cite{AGKEP,Flores,Ohshima,Ozawa,RTV,Vigdorovich}. This note concerns the amenable side, where the study of selflessness combines $\Zst$-stability with the classical question of when a twisted group $C^*$-algebra $\Cr(G,\sigma)$ is simple.

Two remarks put the amenable case in perspective. If $G$ is a nontrivial amenable group, then $\Cr(G)=C^*(G)$ admits a character, so it is never simple and never selfless. Selflessness for amenable groups is therefore a purely twisted phenomenon. It is also genuinely $C^*$-algebraic, since if $G$ is countably infinite amenable and $(G,\sigma)$ satisfies Kleppner's condition, then $W^*(G,\sigma)$ is always the hyperfinite II$_1$ factor, which satisfies the von Neumann analogue of selflessness, see \cite[Remark~2.7]{Robert}.

The arguments combine the following known results. In the separable nuclear setting selflessness is equivalent to simplicity together with a unique tracial state and $\Zst$-stability \cite{Robert,MatuiSato,Ozawa}. For twisted group $C^*$-algebras of amenable groups, a unique tracial state implies simplicity \cite{BedosOmlandST,Murphy}. For FC-hypercentral $G$, a unique tracial state and simplicity are both equivalent to Kleppner's condition \cite{BedosOmlandFC,BedosOmlandST}. Twisted group $C^*$-algebras of virtually polycyclic groups have finite nuclear dimension \cite{EckhardtWu}. For simple, separable, unital, nuclear, infinite-dimensional $C^*$-algebras, finite nuclear dimension is equivalent to $\Zst$-stability \cite{Winter,CETWW}. Thus Kleppner's condition gives $\Zst$-stability when $G$ is infinite, finitely generated, and virtually nilpotent. Outside the finitely generated setting, only $\Zst$-stability is missing. Our first main result supplies it for abelian groups and their finite extensions, without finite generation. Kleppner's condition and its relative version are recalled in Definitions~\ref{d:kleppner} and~\ref{d:relK}.

\begin{mainthm}\label{t:main}
Let $G$ be a countably infinite group with a two-cocycle $\sigma$. Suppose that $G$ is virtually abelian, or finitely generated and virtually nilpotent, or locally finite and FC-hypercentral. Then the following are equivalent:
\begin{itemize}
\item[(i)] $\Cr(G,\sigma)$ is selfless;
\item[(ii)] $(G,\sigma)$ satisfies Kleppner's condition.
\end{itemize}
In this case $\Cr(G,\sigma)$ is $\Zst$-stable and has nuclear dimension at most one. If $G$ is abelian, the nuclear dimension is zero when $G$ is torsion and one otherwise.
\end{mainthm}

The main new ingredient is the abelian $\Zst$-stability theorem in Section~\ref{sec:abelian}. For finitely generated abelian groups, $\Zst$-stability is characterized in \cite[Theorem~4.11]{EnstadVilalta} by the radical of the commutator bicharacter having infinite index. If the commutator bicharacter takes a value of infinite order, that theorem applies to suitable finitely generated subgroups, and $\Zst$-stability passes to $\Cr(G,\sigma)$. If instead every value is a root of unity, then every finitely generated stage is subhomogeneous and fails to absorb $\Zst$. We write $\Cr(G,\sigma)$ as an inductive limit in which the matrix degrees grow faster than the dimensions of the spectra, and we apply the slow dimension growth theorem for simple approximately subhomogeneous algebras \cite{TikuisisRegularity,TikuisisNuclear}. Example~\ref{ex:infdim} shows that simple twisted group $C^*$-algebras of finitely generated subgroups cannot replace this argument. The virtually abelian case reduces to the abelian one by a finite central decomposition and equivariant $\Zst$-stability. The locally finite case is immediate. We conjecture that the equivalence in Theorem~\ref{t:main} holds for every countably infinite FC-hypercentral group, see Remark~\ref{r:scope}.

For inclusions, we also use the theory of $\Zst$-stable inclusions \cite{Sarkowicz} and equivariant $\Zst$-stability for twisted actions of countable amenable groups on simple, nuclear, $\Zst$-stable $C^*$-algebras with a unique tracial state, following \cite{Szabo}. The relative Kleppner condition of \cite{BedosOmlandST,BedosOmlandIrr} governs $C^*$-irreducibility and the relative Dixmier property. With these, the selfless inclusion criterion of \cite{HKEPR} gives our second main result, where the ambient group can be any countable amenable group.

\begin{mainthm}\label{t:am-incl}
Let $G$ be a countable amenable group with a two-cocycle $\sigma$, let $H\trianglelefteq G$ be infinite and write $\sigma'=\sigma|_{H\times H}$. Then the following are equivalent:
\begin{itemize}
\item[(i)] $\Cr(H,\sigma')\subseteq \Cr(G,\sigma)$ is a selfless inclusion;
\item[(ii)] $\Cr(H,\sigma')$ is selfless and $(H\leq G,\sigma)$ satisfies the relative Kleppner condition.
\end{itemize}
When these hold, every intermediate $C^*$-subalgebra is of the form $\Cr(\Gamma,\sigma|_\Gamma)$ for a unique intermediate subgroup $H\leq\Gamma\leq G$ and is selfless.
\end{mainthm}

The main ingredient behind the inclusion results is Proposition~\ref{p:crossed}, which says that $\Zst$-stability of the inclusion needs no outerness or any other condition on the dynamics. The dynamics enter only through the relative Kleppner condition. Thus the relative Kleppner condition alone characterizes selflessness of the inclusion when $H$ belongs to one of the classes in Theorem~\ref{t:main}, see Corollary~\ref{c:incl}. The inclusion results reach beyond the finitely generated virtually nilpotent class, as we illustrate with the lamplighter group and with $\Z\wr\Z$ in Examples~\ref{ex:lamp} and~\ref{ex:zwrz}. Selfless inclusions also arise from reduced amalgamated free products \cite{GKPT} and from commensurators of hyperbolic groups \cite{BasuFlores}. For selfless inclusions of twisted group $C^*$-algebras on the non-amenable side, see \cite{Flores}, where they come from an action-dependent version of the PHP property of \cite{Ozawa}. In Theorem~\ref{t:am-incl} the groups are amenable and the smaller algebra is the coefficient algebra of a cocycle crossed product.

\section{Preliminaries}\label{sec:prelim}

Let $G$ be a discrete group with identity $e$. A \emph{two-cocycle} on $G$ is a map $\sigma\colon G\times G\to\T$ such that
\[
\sigma(g,h)\sigma(gh,k)=\sigma(h,k)\sigma(g,hk)
\quad\text{and}\quad
\sigma(g,e)=\sigma(e,g)=1
\]
for all $g,h,k\in G$. The \emph{$\sigma$-projective left regular representation} of $G$ on $\ell^2(G)$ is given by $\lambda_\sigma(g)\delta_h=\sigma(g,h)\delta_{gh}$ and satisfies $\lambda_\sigma(g)\lambda_\sigma(h)=\sigma(g,h)\lambda_\sigma(gh)$. We write $\Cr(G,\sigma)$ for the reduced twisted group $C^*$-algebra generated by $\{\lambda_\sigma(g)\}_{g\in G}$ and $\tau$ for its canonical faithful tracial state, determined by $\tau(\lambda_\sigma(g))=\delta_{g,e}$. The algebra $\Cr(G,\sigma)$ is unital. It is separable when $G$ is countable, and nuclear if and only if $G$ is amenable, see e.g.\ \cite[Proposition~2.14]{OmlandPrime}. If $G$ is amenable, it is also exact, and the full and reduced twisted group $C^*$-algebras coincide. For a subgroup $K\leq G$, the canonical map $\Cr(K,\sigma|_K)\to\Cr(G,\sigma)$ is a unital trace-preserving embedding.

\begin{definition}[{\cite{Kleppner}}]\label{d:kleppner}
An element $g\in G$ is \emph{$\sigma$-regular} if $\sigma(g,h)=\sigma(h,g)$ for all $h\in G$ with $gh=hg$. One checks that $\sigma$-regularity is a property of conjugacy classes. The pair $(G,\sigma)$ satisfies \emph{Kleppner's condition} if every nontrivial $\sigma$-regular conjugacy class is infinite (equivalently, if the twisted group von Neumann algebra $W^*(G,\sigma)$ is a factor \cite{Kleppner}). Kleppner's condition is necessary but in general not sufficient for $\Cr(G,\sigma)$ to be simple or to have a unique tracial state \cite{BedosOmlandST}.
\end{definition}

The \emph{FC-center} $FC(G)$ of $G$ is the set of elements with finite conjugacy class. The \emph{upper FC-series} $\{F_\alpha\}$ of $G$ is defined by $F_0=\{e\}$, $F_{\alpha+1}/F_\alpha=FC(G/F_\alpha)$, and $F_\lambda=\bigcup_{\alpha<\lambda}F_\alpha$ at a limit ordinal $\lambda$. The series eventually stabilizes, and its final term is called the \emph{FC-hypercenter} of $G$. The group $G$ is \emph{FC-hypercentral} if its FC-hypercenter is all of $G$, equivalently, if $G$ has no nontrivial ICC quotient, see e.g.\ \cite{BedosOmlandFC}. Every FC-hypercentral group is amenable, and every subgroup of an FC-hypercentral group is again FC-hypercentral. Every virtually nilpotent group is FC-hypercentral, and every finitely generated virtually nilpotent group is virtually polycyclic. We say that a group has \emph{polynomial growth} if each of its finitely generated subgroups does. Every FC-hypercentral group has polynomial growth, and for finitely generated groups, being virtually nilpotent, being FC-hypercentral, and having polynomial growth are equivalent \cite{BedosOmlandFC}.

If $G$ is amenable and $\Cr(G,\sigma)$ has a unique tracial state, then $\Cr(G,\sigma)$ is simple, as a consequence of \cite{Murphy}, see \cite[Theorem~2.1]{BedosOmlandST}. If $G$ is FC-hypercentral and $(G,\sigma)$ satisfies Kleppner's condition, then $\Cr(G,\sigma)$ has a unique tracial state \cite[Theorem~3.1]{BedosOmlandFC}. Thus simplicity of $\Cr(G,\sigma)$, uniqueness of its tracial state, and Kleppner's condition for $(G,\sigma)$ are equivalent in this case.

For abelian $G$ we use additive notation. All conjugacy classes are then singletons, so Kleppner's condition states precisely that the \emph{commutator bicharacter} $\beta(x,y)=\sigma(x,y)\overline{\sigma(y,x)}$ is nondegenerate, that is, for every $x\neq0$ there is $y$ with $\beta(x,y)\neq1$. Since abelian groups are FC-hypercentral, nondegeneracy of $\beta$ is equivalent to simplicity of $\Cr(G,\sigma)$ and to uniqueness of its tracial state. We refer to this as the abelian Kleppner criterion. The simplicity statement is classical \cite{Slawny}.

We write $\Zst$ for the Jiang--Su algebra \cite{JiangSu} and call $A$ \emph{$\Zst$-stable} if $A\otimes\Zst\cong A$. The nuclear dimension of a simple, separable, unital $C^*$-algebra belongs to $\{0,1,\infty\}$ \cite[Corollary~C]{CETWW}, and it is zero precisely when the algebra is AF \cite[Remarks~2.2(iii)]{WinterZacharias}. For a free ultrafilter $\omega$ on $\N$, we write $A^{\omega}$ for the norm ultrapower. If $\tau$ is a tracial state on $A$, then $\tau^{\omega}$ denotes the induced tracial state. Finally, $\star$ denotes the reduced free product taken with respect to the GNS representations.

\begin{definition}[{\cite{Robert}}]\label{d:selfless}
A \emph{tracial $C^*$-probability space} is a pair $(A,\tau)$ with $A$ unital and $\tau$ a tracial state. Such a pair, with $A$ separable, $A\neq\C$, and $\tau$ faithful, is \emph{selfless} if the first-factor embedding $(A,\tau)\to(A,\tau)\star(A,\tau)$ into the reduced free product is existential. That is, for some free ultrafilter $\omega$ on $\N$ there is a trace-preserving embedding $(A,\tau)\star(A,\tau)\hookrightarrow(A^{\omega},\tau^{\omega})$ whose restriction to the first factor is the diagonal embedding.
\end{definition}

We always regard $\Cr(G,\sigma)$ as a tracial $C^*$-probability space with its canonical trace $\tau$. Thus ``$\Cr(G,\sigma)$ is selfless'' means that $(\Cr(G,\sigma),\tau)$ is selfless. Similarly, all inclusions of twisted group $C^*$-algebras are taken with respect to the canonical traces, which we suppress from the notation.

Let $A$ be a separable, unital, nuclear, infinite-dimensional $C^*$-algebra with a faithful tracial state. If $A$ is selfless, then it is simple, has a unique tracial state, has strict comparison, and has stable rank one \cite[Theorem~3.1]{Robert}. Conversely, suppose that $A$ is simple with a unique tracial state. Then strict comparison implies selflessness \cite[Theorem~1.1]{RobertSC}, and so does $\Zst$-stability \cite[Theorem~3]{Ozawa}, since $A$ is exact. Moreover, strict comparison implies $\Zst$-stability \cite{MatuiSato}, and $\Zst$-stability is equivalent to finite nuclear dimension \cite{Winter,CETWW}. Thus, for a simple $A$ with a unique tracial state, selflessness and the three regularity properties $\Zst$-stability, finite nuclear dimension, and strict comparison are all equivalent.

\begin{proposition}[Amenable groups]\label{p:am}
Let $G$ be a countably infinite amenable group with a two-cocycle $\sigma$. Then the following are equivalent:
\begin{itemize}
\item[(i)] $\Cr(G,\sigma)$ is selfless;
\item[(ii)] $\Cr(G,\sigma)$ has a unique tracial state and strict comparison;
\item[(iii)] $\Cr(G,\sigma)$ has a unique tracial state and is $\Zst$-stable;
\item[(iv)] $\Cr(G,\sigma)$ has a unique tracial state and has finite nuclear dimension.
\end{itemize}
If $G$ is FC-hypercentral, then uniqueness of the tracial state may be replaced by Kleppner's condition.
\end{proposition}

\begin{proof}
Since $G$ is countably infinite and amenable, the algebra $\Cr(G,\sigma)$ is separable, unital, nuclear, and infinite-dimensional. In (ii), (iii), and (iv) it has a unique tracial state, so it is simple, and the equivalences and the last assertion follow directly from the above.
\end{proof}

\section{Countable abelian groups}\label{sec:abelian}

For finitely generated nilpotent $G$, nowhere scatteredness, pureness, and $\Zst$-stability of $\Cr(G,\sigma)$ are equivalent to a generalized nonrationality condition on $\sigma$ \cite[Theorem~B]{EnstadVilalta}. The proof of $\Zst$-stability realizes $\Cr(G,\sigma)$ as a section algebra over the dual of the subgroup of $\sigma$-regular central elements, and passes $\Zst$-stability from the fibers. This step requires the base space to be finite-dimensional, which is where finite generation is used. In this section we remove finite generation for abelian groups.

\begin{proposition}[Countable abelian groups]\label{p:abelian}
Let $G$ be a countably infinite abelian group with a two-cocycle $\sigma$ satisfying Kleppner's condition. Then $\Cr(G,\sigma)$ is $\Zst$-stable. Its nuclear dimension is zero if $G$ is torsion and one otherwise.
\end{proposition}

For a subgroup $H\leq G$ let
\[
R_H=\{h\in H:\beta(h,k)=1\text{ for all }k\in H\}
\]
be the radical of $\beta|_{H\times H}$, and write $\rk H$ for the torsion-free rank of $H$. Kleppner's condition says that $R_G=\{0\}$. We first record an elementary $K_1$-obstruction to being AF.

\begin{lemma}\label{l:k1gen}
Let $G$ be a countable abelian group with a two-cocycle $\sigma$ and let $x\in G$ have infinite order. Then $[\lambda_\sigma(x)]\neq0$ in $K_1(\Cr(G,\sigma))$. In particular $\Cr(G,\sigma)$ is not an AF algebra.
\end{lemma}

\begin{proof}
Choose finitely generated subgroups $H_1\subseteq H_2\subseteq\cdots$ with $x\in H_1$ and union $G$, and put $A_n=\Cr(H_n,\sigma|_{H_n})$. By continuity of $K$-theory it is enough to show that $[\lambda_\sigma(x)]\neq0$ in $K_1(A_n)$ for every $n$.

Fix $n$. Since $H_n$ is finitely generated abelian and $x$ has infinite order, there is a homomorphism $p\colon H_n\to\Z$ with $j=p(x)\neq0$. It defines a circle action on $A_n$ by
\[
\alpha_z(\lambda_\sigma(h))=z^{p(h)}\lambda_\sigma(h)
\qquad(z\in\T,\ h\in H_n).
\]
If $[\lambda_\sigma(x)]=0$, then for some $k$ there is a path $(w_s)_{s\in[0,1]}$ in $\U(M_k(A_n))$ from $\operatorname{diag}(\lambda_\sigma(x),1,\ldots,1)$ to $1$. The formula
\[
L_s(z)=w_s^*\alpha_z^{(k)}(w_s)
\]
gives a homotopy of loops based at $1$ from $z\mapsto\operatorname{diag}(z^j1_{A_n},1,\ldots,1)$ to the constant loop. Under the standard Bott identification $\pi_1(\U_\infty(A_n))\cong K_0(A_n)$ the former loop represents $j[1_{A_n}]$. This is nonzero because the canonical trace sends it to $j$. The contradiction proves the claim. Since AF algebras have trivial $K_1$, the last assertion follows.
\end{proof}

\begin{lemma}\label{l:stage}
Let $H$ be a finitely generated abelian group with a two-cocycle $\sigma$ whose commutator bicharacter takes only root-of-unity values. Then $R_H$ has finite index in $H$, this index is a square $q^2$, and $\Cr(H,\sigma)$ is $q$-homogeneous with spectrum homeomorphic to $\widehat{R_H}$. The space $\widehat{R_H}$ is a finite disjoint union of tori of dimension $\rk H$.
\end{lemma}

\begin{proof}
Write $R=R_H$ and fix a finite generating set of $H$. The values of $\beta$ on pairs of generators are roots of unity. If $M$ is the least common multiple of their orders, then by bilinearity every value of $\beta$ lies in the group $\mu_M$ of $M$-th roots of unity. The map $h\mapsto\beta(h,\,\cdot\,)$ is a homomorphism from $H$ into the finite group $\operatorname{Hom}(H,\mu_M)$ with kernel $R$, so $R$ has finite index in $H$.

We have $\lambda_\sigma(h)\lambda_\sigma(g)\lambda_\sigma(h)^*=\beta(h,g)\lambda_\sigma(g)$, so conjugation by $\lambda_\sigma(h)$ multiplies the $g$th Fourier coefficient by $\beta(h,g)$, and an element is central if and only if its Fourier coefficients vanish off $R$. With the canonical expectation onto $\Cr(R,\sigma|_R)$ this gives
\[
Z(\Cr(H,\sigma))=\Cr(R,\sigma|_R).
\]
Since $\beta$ is trivial on $R\times R$, the cocycle $\sigma|_R$ is symmetric. The corresponding abelian extension of $R$ by $\T$ splits because $\T$ is divisible, equivalently because $\operatorname{Ext}^1_{\Z}(R,\T)=0$. Hence $\sigma|_R$ is a coboundary and the center is (noncanonically) isomorphic to $C(\widehat R)$.

Put $A=\Cr(H,\sigma)$ and choose a transversal $T\subseteq H$ for $H/R$. The unitaries indexed by $R$ are central and distinct cosets have disjoint Fourier supports. Hence
\[
A=\bigoplus_{t\in T} Z(A)\lambda_\sigma(t)
\]
as a module over $Z(A)$. Thus $A$ is a free $Z(A)$-module of rank $[H:R]$, and every central fiber has vector-space dimension $[H:R]$.

The fiber at $\chi\in\widehat R$ is naturally $\Cr(H/R,\sigma_\chi)$ for a $\chi$-dependent two-cocycle $\sigma_\chi$, see \cite[3.7 and Theorem~3.8]{EnstadVilalta}. Since $H$ is abelian, $\sigma_\chi$ differs from a fixed cocycle on $H/R$ by a symmetric factor. So its commutator bicharacter is the one induced by $\beta$, as in the proof of \cite[Lemma~4.10]{EnstadVilalta}. This bicharacter is nondegenerate by the definition of $R$, so the fiber is simple by the abelian Kleppner criterion. Being finite-dimensional and simple, it is $M_q$ for some $q$, and the dimension count gives $q^2=[H:R]$. Every irreducible representation restricts to a character of the center and therefore factors through exactly one of these fibers. Thus $A$ is $q$-homogeneous and its spectrum is homeomorphic to $\widehat R$.

Finally $R$ is finitely generated abelian, say $R\cong\Z^s\oplus F$ with $F$ finite, so $\widehat R$ is a disjoint union of $|F|$ copies of $\T^s$. Since $R$ has finite index in $H$ we get $s=\rk R=\rk H$.
\end{proof}

In particular no finitely generated stage is $\Zst$-stable in this case, since it is subhomogeneous. Lemma~\ref{l:stage} instead reduces the root-of-unity case to making the matrix degree outgrow the spectral dimension. The next lemma chooses an increasing sequence of finitely generated subgroups for which this happens.

\begin{lemma}\label{l:exhaust}
Let $G$ be a countably infinite abelian group with a two-cocycle $\sigma$ satisfying Kleppner's condition. Assume that every value of $\beta$ is a root of unity and that $G$ contains an element of infinite order. Then there are finitely generated subgroups $H_1\subseteq H_2\subseteq\cdots$ with union $G$ such that
\[
\rk H_n\leq\frac{q(H_n)}{n},
\qquad\text{where}\quad q(H_n)^2=[H_n:R_{H_n}].
\]
\end{lemma}

\begin{proof}
Fix $x\in G$ of infinite order. The orders of the numbers $\beta(x,y)$ for $y\in G$ are unbounded. Otherwise their least common multiple $M$ would satisfy $\beta(Mx,y)=1$ for every $y$, while $Mx\neq0$, contradicting Kleppner's condition.

Enumerate $G=\{g_1,g_2,\dots\}$ and put $H_0=\langle x\rangle$. Given $H_{n-1}$, set $K_n=\langle H_{n-1},g_n\rangle$ and $r_n=\rk K_n$. Choose $y_n\in G$ with $m_n=\ord\beta(x,y_n)\geq n(r_n+1)$ and put $H_n=\langle K_n,y_n\rangle$. The subgroups $H_n$ increase, their union is $G$, and $\rk H_n\leq r_n+1$.

Write $R=R_{H_n}$. The $m_n^2$ cosets $ax+by_n+R$ with $0\leq a,b<m_n$ are pairwise distinct. Indeed, pairing a relation between two of them with $x$ and with $y_n$ shows that both coefficient differences are divisible by $m_n$. Hence $[H_n:R]\geq m_n^2$, and Lemma~\ref{l:stage} gives $q(H_n)\geq m_n\geq n(r_n+1)\geq n\rk H_n$.
\end{proof}

\begin{proof}[Proof of Proposition~\ref{p:abelian}]
Write $A=\Cr(G,\sigma)$. It is separable, unital, nuclear, and infinite-dimensional, and by the abelian Kleppner criterion it is simple with a unique tracial state.

Suppose first that every element of $G$ has finite order. Then every finitely generated subgroup of $G$ is finite, so $A$ is an inductive limit of finite-dimensional twisted group $C^*$-algebras and hence AF. Thus $A$ has nuclear dimension zero and is $\Zst$-stable by the results recalled in Section~\ref{sec:prelim}.

Suppose instead that $G$ contains an element of infinite order. By Lemma~\ref{l:k1gen} the algebra $A$ is not AF, so by the regularity results recalled in Section~\ref{sec:prelim} it suffices to prove that $A$ is $\Zst$-stable. The nuclear dimension is then one. There are two cases.

Assume that some value of $\beta$ has infinite order, say $\beta(x_0,y_0)$. Choose finitely generated subgroups $H_1\subseteq H_2\subseteq\cdots$ with union $G$ and with $x_0,y_0\in H_1$. For each $n$ the cosets $kx_0+R_{H_n}$ with $k\in\Z$ are pairwise distinct, because $\beta(kx_0,y_0)=\beta(x_0,y_0)^k$. So $R_{H_n}$ has infinite index in $H_n$, and $\Cr(H_n,\sigma|_{H_n})$ is $\Zst$-stable by \cite[Theorem~4.11]{EnstadVilalta}. Hence $A$ is $\Zst$-stable by \cite[Corollary~3.4]{TomsWinterSSA}.

Assume instead that every value of $\beta$ is a root of unity, and take the subgroups $H_1\subseteq H_2\subseteq\cdots$ of Lemma~\ref{l:exhaust}. Put $A_n=\Cr(H_n,\sigma|_{H_n})$ and $q_n=q(H_n)$. By Lemma~\ref{l:stage} each $A_n$ is $q_n$-homogeneous with compact finite-dimensional spectrum, and the dimension--rank ratio of $1_{A_n}$ in the sense of \cite[Definition~5.1]{TikuisisRegularity} is
\[
\frac{\dim\widehat{R_{H_n}}}{q_n}
=\frac{\rk H_n}{q_n}\leq\frac1n.
\]
The canonical connecting maps are injective and unital, so each image generates the target as an ideal. So they are full in the sense of \cite[Proposition~3.6]{TikuisisRegularity}. Their inductive limit is $A$, which is simple, and the image of $1_{A_1}$ in $A_n$ is $1_{A_n}$. Thus the preceding display verifies condition~(ii) of \cite[Proposition~5.2]{TikuisisRegularity}. By that proposition and \cite[Definition~5.3]{TikuisisRegularity}, $A$ is approximately subhomogeneous with slow dimension growth. Hence $A$ is $\Zst$-stable by \cite[Corollary~9.2]{TikuisisNuclear}.
\end{proof}

\begin{example}[Infinite-dimensional noncommutative tori]\label{ex:infdim}
Consider the free abelian group of countably infinite rank $G=\bigoplus_{j\geq1}\Z$ with generators $(e_j)_{j\geq1}$, and write $x=(x_j)_{j\geq1}$. Let $\sigma_\Sigma$ be the antisymmetric two-cocycle
\[
\sigma_\Sigma(x,y)=\exp\!\left(\pi i\sum_{j<k}\frac{x_jy_k-x_ky_j}{j+k}\right),
\]
so that $\lambda_{\sigma_\Sigma}(e_j)\lambda_{\sigma_\Sigma}(e_k)=e^{2\pi i/(j+k)}\lambda_{\sigma_\Sigma}(e_k)\lambda_{\sigma_\Sigma}(e_j)$ for $j<k$. Even though each value of $\sigma_\Sigma$ is a root of unity, $(G,\sigma_\Sigma)$ satisfies Kleppner's condition. Indeed, a nonzero $x=\sum_{j\leq N}a_je_j$ is $\sigma_\Sigma$-regular only if $\sum_{j\leq N}a_j/(j+k)\in\Z$ for all $k>N$. The sum is a rational function of $k$ which tends to $0$ as $k\to\infty$. It is not identically zero, since the functions $(k+j)^{-1}$ for $j\leq N$ are linearly independent. Hence, for all sufficiently large $k$, it lies in $(-1,1)\setminus\{0\}$, and therefore is not an integer. Thus no nonzero $x$ is $\sigma_\Sigma$-regular, and Propositions~\ref{p:abelian} and~\ref{p:am} show that $\Cr(G,\sigma_\Sigma)$ is $\Zst$-stable and selfless, and has nuclear dimension one.

For nontrivial finitely generated torsion-free abelian groups, a cocycle with values in the roots of unity is incompatible with Kleppner's condition. Indeed, a nontrivial finitely generated subgroup $L$ of $G$ is infinite, so $R_L$ is nonzero by Lemma~\ref{l:stage}, and $(L,\sigma_\Sigma|_L)$ fails Kleppner's condition. Thus no nontrivial finitely generated subgroup of $G$ has a simple twisted group $C^*$-algebra. That a root-of-unity-valued cocycle could satisfy Kleppner's condition when $G$ is not finitely generated was also observed in \cite[Example~5.9]{BedosOmlandST}, and earlier for the noncommutative solenoids of \cite{LatremolierePacker}.
\end{example}

\section{\texorpdfstring{$\Zst$}{Z}-stable inclusions and finite extensions}\label{sec:extensions}

Let $H\trianglelefteq G$ and write $\sigma'=\sigma|_{H\times H}$. The canonical inclusion $\Cr(H,\sigma')\subseteq \Cr(G,\sigma)$ is unital and trace-preserving. Moreover, with $K=G/H$, there is an isomorphism $\Cr(G,\sigma)\cong \Cr(H,\sigma')\rtimes_{r,(\alpha,u)} K$, where $(\alpha,u)$ is the canonical twisted action of $K$ on $\Cr(H,\sigma')$, see \cite{Bedos,BedosOmlandST}.

\begin{definition}[{\cite{Sarkowicz}}]\label{d:zstincl}
A unital inclusion $B\subseteq A$ of separable $C^*$-algebras is a \emph{$\Zst$-stable inclusion} if there is an isomorphism $A\to A\otimes\Zst$ that restricts to an isomorphism $B\to B\otimes\Zst$.
\end{definition}

By \cite[Theorem~4.4]{Sarkowicz}, a unital inclusion $B\subseteq A$ of separable $C^*$-algebras is a $\Zst$-stable inclusion if and only if there is a unital embedding $\Zst\hookrightarrow B^{\omega}\cap A'$. Here the relative commutant is taken in $A^{\omega}$, $B^{\omega}$ is identified with its coordinatewise image in $A^{\omega}$, and $A$ is embedded as the constant sequences. The main tool of this section is equivariant $\Zst$-stability for twisted actions on algebras with a unique tracial state. For genuine actions, such results were obtained in \cite{Sato,GHV}.

\begin{proposition}[$\Zst$-stability of the inclusion]\label{p:crossed}
Let $B$ be a simple, separable, unital, nuclear, $\Zst$-stable $C^*$-algebra with a unique tracial state. For every twisted action $(\alpha,u)$ of a countable amenable group $K$ on $B$, the inclusion $B\subseteq B\rtimes_{r,(\alpha,u)}K$ is a $\Zst$-stable inclusion.
\end{proposition}

\begin{proof}
We verify the hypotheses of \cite[Theorem~5.20]{Szabo} for $B$. The algebra $B$ is separable, simple, and nuclear. It is non-elementary because it is unital, $\Zst$-stable, and hence infinite-dimensional, and it is finite because it is unital with a faithful trace. Being $\Zst$-stable, it has strict comparison, hence the very weak comparison required in \cite{Szabo}. The trace is the only tracial state, so $B$ has a single ray of extremal traces. Under these hypotheses, \cite[Theorem~5.20]{Szabo} gives equivariant $\Zst$-stability for every action of $K$ on $B$, and by \cite[Remark~5.25]{Szabo} the same holds for cocycle actions. Since $u(k,\ell)\in\U(B)$, the inner automorphism $\operatorname{Ad}u(k,\ell)$ acts trivially on the central sequence algebra $B^{\omega}\cap B'$. So $(\alpha,u)$ induces a genuine action of $K$ on $B^{\omega}\cap B'$, and the proof of \cite[Theorem~5.20]{Szabo} gives a unital embedding $\Zst\hookrightarrow (B^{\omega}\cap B')^{K}$.

Put $A=B\rtimes_{r,(\alpha,u)}K$. The algebra $A$ is generated by $B$ and the implementing unitaries $\{\lambda_k\}_{k\in K}$. Hence an element $x\in B^{\omega}$ commutes with $A$ if and only if it commutes with $B$ and with each $\lambda_k$. Since $\lambda_k x\lambda_k^*=\alpha_k^{\omega}(x)$, this means that $x\in B^{\omega}\cap B'$ and $\alpha_k^{\omega}(x)=x$ for all $k$. The cocycle adds no condition because $x$ commutes with $B$. Thus
\[
B^{\omega}\cap A'=(B^{\omega}\cap B')^{K},
\]
and the embedding above shows that $B\subseteq A$ is a $\Zst$-stable inclusion.
\end{proof}

No outerness of the action is needed for Proposition~\ref{p:crossed}. The unique tracial state of $B$ makes the trace simplex a point and removes any orbit condition on the traces. By the decomposition above, the proposition applies to $\Cr(H,\sigma')\subseteq\Cr(G,\sigma)$ whenever $G$ is countable amenable and $\Cr(H,\sigma')$ is simple, $\Zst$-stable, and has a unique tracial state.

\begin{proposition}[Virtually abelian groups]\label{p:va}
Let $G$ be a countably infinite virtually abelian group. If $(G,\sigma)$ satisfies Kleppner's condition, then $\Cr(G,\sigma)$ is $\Zst$-stable.
\end{proposition}

\begin{proof}
Take the core of an abelian subgroup of finite index to obtain an infinite abelian normal subgroup $N\trianglelefteq G$ of finite index. Put
\[
A=\Cr(G,\sigma),\quad B=\Cr(N,\sigma|_N),\quad Q=G/N,
\]
so $A\cong B\rtimes_{r,(\alpha,u)}Q$, with implementing unitaries $V_q$ for $q\in Q$. Since $G$ is FC-hypercentral and $(G,\sigma)$ satisfies Kleppner's condition, $A$ is simple, see Section~\ref{sec:prelim}. Let $R$ be the radical of the commutator bicharacter on $N$. The Fourier coefficient argument of Lemma~\ref{l:stage}, which does not require finite generation for this assertion, gives
\[
Z(B)=\Cr(R,\sigma|_R)\cong C(\widehat R).
\]
The action of $Q$ on $Z(B)$ is genuine and $Z(B)^Q\subseteq Z(A)=\C$. Thus $\widehat R/Q$ is a point. Since $Q$ is finite, $\widehat R$ is finite and $Q$ acts transitively on it. In particular $R$ is finite.

For each minimal central projection $p_\chi\in Z(B)$ the fiber description in \cite[3.7 and Theorem~3.8]{EnstadVilalta} gives
\[
p_\chi B\cong\Cr(N/R,\omega_\chi).
\]
The commutator bicharacter of $\omega_\chi$ is induced by $\beta$. Indeed, the additional character-dependent factor comes from the abelian extension $R\to N\to N/R$ and is symmetric. The induced bicharacter is nondegenerate by the definition of $R$. Since $N/R$ is countably infinite, Proposition~\ref{p:abelian} shows that every $p_\chi B$ is simple, monotracial, and $\Zst$-stable.

Fix $p=p_\chi$, set $D=pB$, and let $Q_p=\{q\in Q:\alpha_q(p)=p\}$. If $q_1,\ldots,q_m$ represent $Q/Q_p$, the partial isometries $s_i=V_{q_i}p$ satisfy
\[
s_i^*s_i=p,\qquad s_is_i^*=\alpha_{q_i}(p),\qquad
\sum_i s_is_i^*=1.
\]
They give $A\cong M_m(pAp)$. Furthermore
$pbV_qp=pb\alpha_q(p)V_q$ vanishes for $q\notin Q_p$. Therefore
\[
pAp\cong D\rtimes_{r,(\gamma,w)}Q_p,
\qquad \gamma_q=\alpha_q|_D,\quad w(q,r)=pu(q,r).
\]
Indeed, the indicated generators give a surjection which intertwines the faithful coefficient expectations and restricts to the identity on $D$, so it is injective. By Proposition~\ref{p:crossed} the algebra $pAp$ is $\Zst$-stable, and hence so is $A\cong M_m(pAp)$.
\end{proof}

\begin{proof}[Proof of Theorem~\ref{t:main}]
In each of the three cases $G$ is FC-hypercentral, so selflessness implies Kleppner's condition by Proposition~\ref{p:am}. Conversely, suppose that Kleppner's condition holds. If $G$ is virtually abelian, then $\Cr(G,\sigma)$ is $\Zst$-stable by Proposition~\ref{p:va}. If $G$ is finitely generated and virtually nilpotent, then $G$ is virtually polycyclic, so $\Cr(G,\sigma)$ has finite nuclear dimension by \cite[Theorem~A]{EckhardtWu}. If $G$ is locally finite, then $\Cr(G,\sigma)$ is an inductive limit of finite-dimensional twisted group $C^*$-algebras and hence AF. In all three cases Proposition~\ref{p:am} gives selflessness. The nuclear dimension assertions follow from \cite[Corollary~C]{CETWW} and Proposition~\ref{p:abelian}.
\end{proof}

\section{Selfless inclusions}\label{sec:inclusions}

We now add the dynamical condition that turns the $\Zst$-stable inclusion of Section~\ref{sec:extensions} into a selfless inclusion, namely the relative Kleppner condition. We first recall the notions involved.

\begin{definition}[{\cite{HKEPR}}]\label{d:selfincl}
Let $(A,\tau)$ be a tracial $C^*$-probability space with $A$ separable and $\tau$ faithful, and let $B\subseteq A$ be a unital $C^*$-subalgebra, regarded as a tracial $C^*$-probability space with $\tau|_B$. The inclusion $B\subseteq(A,\tau)$ is a \emph{selfless inclusion} if there is a $C^*$-probability space $(C,\kappa)$ with $C\neq\C$ such that the first-factor embedding of inclusions $(B\subseteq A)\to(B\star C\subseteq A\star C)$ is existential. Equivalently, for some free ultrafilter $\omega$ on $\N$ there is a state-preserving embedding $(A,\tau)\star(C,\kappa)\hookrightarrow(A^{\omega},\tau^{\omega})$ restricting to the diagonal embedding on $A$ and carrying $B\star C$ into $B^{\omega}$. Here $\kappa$ is a state on $C$ that induces a faithful GNS representation.
\end{definition}

If $B\subseteq(A,\tau)$ is a selfless inclusion, then every intermediate $C^*$-subalgebra $B\subseteq D\subseteq A$ is selfless (with $\tau|_D$) \cite[Theorem~2.5(i)]{HKEPR}. In particular, both $B$ and $A$ are selfless, and $B\subseteq A$ is \emph{$C^*$-irreducible}, that is, every intermediate $C^*$-algebra is simple \cite{RordamIrr}. We also use the \emph{relative Dixmier property}, which here means that for every $a\in A$ the norm-closed convex hull of $\{vav^*:v\in\U(B)\}$ meets $\C1$.

\begin{definition}[{\cite[Definition~3.6]{BedosOmlandIrr}}]\label{d:relK}
Let $H\leq G$. For $g\in G$, write $g^{H}=\{\,hgh^{-1}:h\in H\,\}$ for its $H$-conjugacy class, and call $g$ \emph{$\sigma$-regular relative to $H$} if $\sigma(g,h)=\sigma(h,g)$ for all $h\in H$ with $gh=hg$. One checks that this is a property of $H$-conjugacy classes. The triple $(H\leq G,\sigma)$ satisfies the \emph{relative Kleppner condition} if every nontrivial $H$-conjugacy class in $G$ that is $\sigma$-regular relative to $H$ is infinite. This condition incorporates Kleppner's condition for $(H,\sigma|_H)$ and implies it for $(G,\sigma)$.
\end{definition}

We follow the convention of \cite[Definition~3.6]{BedosOmlandIrr}. In \cite[Definition~4.5]{BedosOmlandST} only $H$-conjugacy classes in $G\setminus H$ that are $\sigma$-regular relative to $H$ are required to be infinite, so Kleppner's condition for $(H,\sigma')$ is not built in. The two conventions agree once $(H,\sigma')$ satisfies Kleppner's condition.

Suppose that $H$ is normal. Then the inclusion $\Cr(H,\sigma')\subseteq\Cr(G,\sigma)$ is $C^*$-irreducible if and only if $\Cr(H,\sigma')$ is simple and $(H\leq G,\sigma)$ satisfies the relative Kleppner condition \cite[Theorem~6.2]{BedosOmlandIrr}. If $\Cr(H,\sigma')$ has a unique tracial state, these conditions are also equivalent to the inclusion having the relative Dixmier property. Whenever these equivalent conditions hold, the map $\Gamma\mapsto\Cr(\Gamma,\sigma|_\Gamma)$ is a bijection between the intermediate subgroups $H\leq\Gamma\leq G$ and the intermediate $C^*$-subalgebras of $\Cr(H,\sigma')\subseteq\Cr(G,\sigma)$.

Finally, \cite[Theorem~2.9]{HKEPR} says the following. Let $B\subseteq(A,\rho)$ be an inclusion of $C^*$-probability spaces. If $B\subseteq A$ is a $\Zst$-stable inclusion with the relative Dixmier property, $A$ is exact, and $\rho$ is invariant under conjugation by the unitaries of $B$, then $B\subseteq(A,\rho)$ is a selfless inclusion.

\begin{proof}[Proof of Theorem~\ref{t:am-incl}]
Put $A=\Cr(G,\sigma)$ and $B=\Cr(H,\sigma')$.

(i)$\Longrightarrow$(ii): If $B\subseteq A$ is a selfless inclusion, then $B$ is selfless and the inclusion is $C^*$-irreducible. In particular, $B$ is simple, so the $C^*$-irreducibility of $B\subseteq A$ is equivalent to the relative Kleppner condition for $(H\leq G,\sigma)$. This gives (ii).

(ii)$\Longrightarrow$(i): Since $B$ is selfless, it is simple. Moreover, $H$ is countably infinite and amenable, so $B$ has a unique tracial state and is $\Zst$-stable by Proposition~\ref{p:am}. By Proposition~\ref{p:crossed}, applied to the canonical twisted action of $G/H$ on $B$, the inclusion $B\subseteq A$ is a $\Zst$-stable inclusion. Since $G$ is amenable, $A$ is nuclear and hence exact. The canonical trace $\tau$ is $\U(B)$-invariant because it is a trace. Since $(H\leq G,\sigma)$ satisfies the relative Kleppner condition, \cite[Theorem~6.2]{BedosOmlandIrr} shows that the inclusion $B\subseteq A$ has the relative Dixmier property. All hypotheses of \cite[Theorem~2.9]{HKEPR} hold, so $B\subseteq A$ is a selfless inclusion. Every intermediate $C^*$-subalgebra is then selfless, and it is of the stated form by \cite[Theorem~6.2]{BedosOmlandIrr}.
\end{proof}

The relative Kleppner condition plays the same role for the inclusion that Kleppner's condition plays for the individual algebra.

\begin{remark}[Normality]\label{r:normal}
The implication (i)$\Longrightarrow$(ii) of Theorem~\ref{t:am-incl} holds for every infinite subgroup $H\leq G$. Indeed, a selfless inclusion is $C^*$-irreducible. Its relative commutant $\Cr(H,\sigma')'\cap\Cr(G,\sigma)$ is therefore trivial, since a nonscalar self-adjoint element $x$ of it would be central in the intermediate $C^*$-algebra $C^*(\Cr(H,\sigma'),x)$, which would then not be simple. For an arbitrary subgroup, triviality of the relative commutant is equivalent to the relative Kleppner condition \cite[Proposition~4.3]{BedosOmlandIrr}.
\end{remark}

\begin{corollary}\label{c:incl}
Let $G$ be a countable amenable group with a two-cocycle $\sigma$ and let $H\trianglelefteq G$ be infinite. Suppose that $H$ is virtually abelian, or finitely generated and virtually nilpotent, or locally finite and FC-hypercentral. Then the following are equivalent:
\begin{itemize}
\item[(i)] $\Cr(H,\sigma')\subseteq \Cr(G,\sigma)$ is a selfless inclusion;
\item[(ii)] $(H\leq G,\sigma)$ satisfies the relative Kleppner condition.
\end{itemize}
When these hold, every intermediate $C^*$-subalgebra is of the form $\Cr(\Gamma,\sigma|_\Gamma)$ for a unique intermediate subgroup $H\leq\Gamma\leq G$ and is selfless.
\end{corollary}

\begin{proof}
Under the relative Kleppner condition, $(H,\sigma')$ satisfies Kleppner's condition (Definition~\ref{d:relK}), so $\Cr(H,\sigma')$ is selfless by Theorem~\ref{t:main}. The equivalence with~(i), together with the final statement, follows from Theorem~\ref{t:am-incl}.
\end{proof}

\section{Examples and scope}\label{sec:examples}

\begin{remark}[Noncommutative tori and nilpotent groups]\label{r:nct}
For $G=\Z^n$ with $\sigma_\Theta(x,y)=e^{\pi i\langle x,\Theta y\rangle}$ and $\Theta$ real antisymmetric, Kleppner's condition is satisfied if and only if $\Theta x\notin\Z^n$ for every nonzero $x\in\Z^n$, which is the classical simplicity criterion for the noncommutative torus $A_\Theta=\Cr(\Z^n,\sigma_\Theta)$ \cite{Slawny}. Thus Theorem~\ref{t:main} says that every simple noncommutative torus is selfless. For concrete examples of twisted group $C^*$-algebras of nilpotent groups with explicit simplicity criteria in terms of $\sigma$, see \cite{OmlandFreeNilpotent}.
\end{remark}

We close with a family of noncommutative orbifolds and two wreath products over $\Z$. The wreath-product examples show that our inclusion results also apply to amenable groups outside the finitely generated virtually nilpotent class.

\begin{example}[Noncommutative orbifolds]\label{ex:orb}
Let $\sigma_\theta(x,y)=e^{\pi i\theta(x_1y_2-x_2y_1)}$ be the standard two-cocycle on $\Z^2$, so that $\Cr(\Z^2,\sigma_\theta)=A_\theta$ is the rotation algebra. Let $R\in\mathrm{SL}(2,\Z)$ have finite order $k\geq2$, necessarily $k\in\{2,3,4,6\}$, and form $G=\Z^2\rtimes_R\Z_k$ with $H=\Z^2$. Since $R$ preserves the form $x_1y_2-x_2y_1$, the formula
\[
\sigma\big((x,s),(y,t)\big)=\sigma_\theta(x,R^sy)
\]
defines a two-cocycle on $G$ that restricts to $\sigma_\theta$ on $H$. Moreover, $\Cr(G,\sigma)\cong A_\theta\rtimes_r\Z_k$, where the canonical action is given by $\alpha_s(\lambda_{\sigma_\theta}(x))=\lambda_{\sigma_\theta}(R^sx)$. This crossed product was studied in \cite{ELPW}.

We claim that $(H\leq G,\sigma)$ satisfies the relative Kleppner condition if and only if $\theta$ is irrational. Suppose $\theta$ is irrational. For $x\in\Z^2$, the $H$-conjugacy class of $(x,0)$ is a singleton. If $(x,0)$ is $\sigma$-regular relative to $H$, then $\theta(x_1y_2-x_2y_1)\in\Z$ for every $y\in\Z^2$, which, by irrationality of $\theta$, forces $x=0$. For $(x,s)$ with $s\neq0$, the matrix $R^s\neq I$ has finite order. If $1$ were an eigenvalue of $R^s$, then both eigenvalues would be $1$ since $\det R^s=1$, and $R^s$ would be $I$ since it is diagonalizable. Thus $I-R^s$ is injective on $\Z^2$, and hence the $H$-conjugacy class $\{(x+(I-R^s)y,\,s):y\in\Z^2\}$ is infinite. Conversely, if $\theta$ is rational, then already $(H,\sigma_\theta)$ fails Kleppner's condition.

By Corollary~\ref{c:incl}, the inclusion $A_\theta\subseteq A_\theta\rtimes_r\Z_k$ is selfless if and only if $\theta$ is irrational. For irrational $\theta$, its intermediate $C^*$-algebras correspond to the subgroups of $\Z_k$, so for $k\in\{2,3\}$ there is no $C^*$-algebra strictly between $A_\theta$ and $A_\theta\rtimes_r\Z_k$. Moreover, $A_\theta\rtimes_r\Z_k$ is AF \cite{ELPW}, while $A_\theta$ has nuclear dimension one. Thus the algebras on the two sides of a selfless inclusion can have very different $K$-theory and nuclear dimension.
\end{example}

The last two examples are wreath products. Let $N$ be a nontrivial countable abelian group, let $H=\bigoplus_{j\in\Z}N$ with the shift $s(v)_j=v_{j-1}$, and let $G=N\wr\Z=H\rtimes\Z$, so that $(v,i)(w,j)=(v+s^iw,\,i+j)$. If $\sigma'$ is a shift-invariant two-cocycle on $H$, then
\[
\sigma\big((v,i),(w,j)\big)=\sigma'(v,s^iw)
\]
defines a two-cocycle on $G$ whose restriction to $H$ is $\sigma'$, see \cite[Section~5.2]{BedosOmlandST}. Such a $\sigma$ satisfies the relative Kleppner condition for $(H\leq G,\sigma)$ whenever $(H,\sigma')$ satisfies Kleppner's condition. Indeed, fix $(v,i)$ with $i\neq0$. Its $H$-conjugates are $(v+w-s^iw,\,i)$ for $w\in H$. The map $w\mapsto w-s^iw$ is injective, since a nonzero element fixed by $s^i$ would have infinite support. Hence the conjugacy class is infinite. If $v\neq0$, the class of $(v,0)$ is a singleton. Moreover, $\sigma\big((v,0),(w,0)\big)=\sigma'(v,w)$, so $(v,0)$ is $\sigma$-regular relative to $H$ exactly when $v$ is $\sigma'$-regular in $H$.

\begin{example}[The lamplighter group]\label{ex:lamp}
Take $N=\Z_2$. Then $G=\Z_2\wr\Z$ is the lamplighter group and $H=\bigoplus_\Z\Z_2$. For $j\in\Z$, let $e_j\in H$ be the element supported at $j$ with value $1$. Define $\sigma'(v,w)=(-1)^{\sum_{j<k}v_jw_k}$, so that $\beta(e_j,e_k)=-1$ for all $j\neq k$. The cocycle $\sigma'$ is shift-invariant, and the radical of its commutator bicharacter $\beta$ is trivial. Indeed, if $\beta(v,e_k)=1$ for all $k$, then $\sum_jv_j\equiv v_k$ modulo $2$ for every $k$, which, for finitely supported $v$, forces $v=0$. Thus $(H,\sigma')$ satisfies Kleppner's condition and $(H\leq G,\sigma)$ satisfies the relative Kleppner condition. For $n\geq1$, let $V_n=\langle e_{-n},\dots,e_{n-1}\rangle$. The same equations, summed over the $2n$ coordinates of $V_n$, force their common value to be zero, so $\beta|_{V_n}$ is nondegenerate. Hence $\Cr(V_n,\sigma'|_{V_n})\cong M_{2^n}$, and $\Cr(H,\sigma')$ is the UHF algebra of type $2^{\infty}$, in accordance with \cite[Section~5.2.2]{BedosOmlandST}.

The UHF algebra has a unique tracial state and is $\Zst$-stable, so Theorem~\ref{t:am-incl} applies. The inclusion $\Cr(H,\sigma')\subseteq\Cr(G,\sigma)$ is selfless, and hence $\Cr(G,\sigma)$ is selfless. Here $G$ is amenable and finitely generated but not virtually nilpotent, so $\Cr(G,\sigma)$ is covered neither by Theorem~\ref{t:main} nor by \cite{EckhardtWu}.
\end{example}

\begin{example}[The wreath product $\Z\wr\Z$]\label{ex:zwrz}
Take $N=\Z$, so that $G=\Z\wr\Z$ and $H=\bigoplus_{j\in\Z}\Z$, and let $(e_j)_{j\in\Z}$ denote its standard basis. By \cite[Section~5.2.1]{BedosOmlandST}, every two-cocycle on $\Z\wr\Z$ is cohomologous to the extension of a shift-invariant cocycle on $H$. Let $p_1<p_2<\cdots$ be the primes and put
\[
\sigma'(v,w)=\exp\!\left(2\pi i\sum_{j<k}\frac{v_jw_k}{p_{k-j}}\right).
\]
Then $\sigma'$ is shift-invariant, since the exponent for a pair $j<k$ depends only on $k-j$. This cocycle was considered in \cite[Example~5.9(d)]{BedosOmlandST}.

The pair $(H,\sigma')$ satisfies Kleppner's condition. Indeed, let $v\neq0$, set $l=\max\supp(v)$, and choose $n$ with $p_n>|v_l|$. Then
\[
\beta(v,e_{l+n})=\exp\!\left(2\pi i\sum_{j\leq l}\frac{v_j}{p_{n+l-j}}\right)\neq1.
\]
In this finite sum the prime $p_n$ occurs as a denominator only in the term $v_l/p_n$. Multiplication by the product of the distinct denominator primes and reduction modulo $p_n$ show that the sum cannot be an integer. Hence $(H\leq G,\sigma)$ satisfies the relative Kleppner condition, and Corollary~\ref{c:incl} shows that the inclusion $\Cr(H,\sigma')\subseteq\Cr(G,\sigma)$ is selfless. In particular, $\Cr(G,\sigma)$ is selfless. The intermediate $C^*$-algebras are $\Cr(H\rtimes n\Z,\sigma|_{H\rtimes n\Z})$ for $n\geq1$, together with the base algebra $\Cr(H,\sigma')$. All of them are selfless. The uniqueness of the tracial state on $\Cr(G,\sigma)$ was proved in \cite[Proposition~5.8]{BedosOmlandST}, so the new conclusions here are $\Zst$-stability and selflessness. Here $G$ is a finitely generated, torsion-free, amenable group of exponential growth.

The commutator bicharacter on $H$ is root-of-unity-valued, and every nontrivial finitely generated subgroup of $H$ has a nonzero radical of finite index by Lemma~\ref{l:stage}. So the restriction of $\sigma'$ to such a subgroup never satisfies Kleppner's condition. This example uses the part of Proposition~\ref{p:abelian} that cannot be obtained from inductive limits of simple noncommutative tori.

The base algebra has nuclear dimension one by Proposition~\ref{p:abelian}, and so does the ambient algebra. The latter is simple and $\Zst$-stable by Theorem~\ref{t:am-incl}. To see that it is not AF, write it as $B\rtimes_s\Z$ with $B=\Cr(H,\sigma')$. The Pimsner--Voiculescu boundary map sends the $K_1$-class of the implementing unitary to $\pm[1_B]$, which is nonzero under the canonical trace. Thus $K_1(B\rtimes_s\Z)\neq0$.
\end{example}

\begin{remark}[Scope]\label{r:scope}
For countably infinite FC-hypercentral groups, Proposition~\ref{p:am} reduces the question whether Kleppner's condition implies selflessness to $\Zst$-stability, or equivalently to strict comparison. Theorem~\ref{t:main} settles it for three classes. Whether the implication holds for all countably infinite FC-hypercentral groups remains open. Without simplicity, no general characterization of $\Zst$-stability of $\Cr(G,\sigma)$ is known. For arbitrary abelian groups, nowhere scatteredness is equivalent to $[G:R_G]=\infty$ \cite[Lemma~4.10 and Remark~4.12]{EnstadVilalta}, and $\Zst$-stability requires it. The slow dimension growth theorem used above is not available for nonsimple limits. Finite generation also matters for the nuclear dimension. If $G$ is finitely generated virtually abelian and every value of $\sigma$ is a root of unity, then the nuclear dimension of $\Cr(G,\sigma)$ equals the torsion-free rank of an abelian subgroup of finite index \cite[Theorem~A]{GKM}. Example~\ref{ex:infdim} shows that this equality fails without finite generation, since there the rank is infinite and the nuclear dimension is one.
\end{remark}

\subsection*{AI statement}
The author used various versions of Claude and ChatGPT when obtaining the proofs of Propositions~\ref{p:abelian} and~\ref{p:va}, and otherwise to help check references and improve the exposition. The author has checked all arguments and takes full responsibility for the content.

\end{document}